\documentclass[11pt]{birkjour}

\usepackage{amssymb,amsthm,enumerate}
\usepackage{graphicx,color}

\newcommand {\sem}[1]{\mbox{$\left (\e^{t{#1}}\right )_{t \ge 0}$}}
\newcommand{\e}{\mathrm {e}}
\newcommand{\eps}{\epsilon}
\newcommand{\ud}{\, \mathrm{d}}

\def\RR{{\mathbb{R}}}

\def\CC{{\mathbb{C}}}
\def\cL{{\mathcal{L}}}
\def\cA{{\mathcal{A}}}
\def\1{{\mathbf{1}}}

\newtheorem{theorem}{Theorem}[section]
\newtheorem{corollary}[theorem]{Corollary}
\newtheorem{lemma}[theorem]{Lemma}

\newtheorem{problem}[theorem]{Open Problem}
\theoremstyle{definition}
\newtheorem{definition}[theorem]{Definition}

\newtheorem{example}[theorem]{Example}
\newtheorem{remark}[theorem]{Remark}

\begin{document}

\title{On shape preserving semigroups}
\author[A. B\'{a}tkai]{Andr\'{a}s B\'{a}tkai}
\address{E\"{o}tv\"{o}s Lor\'{a}nd University, Institute of Mathematics, 1117 Budapest, P\'{a}zm\'{a}ny P. s\'{e}t\'{a}ny 1/C, Hungary.}
\email{batka@cs.elte.hu}
\thanks{Supported by the Alexander von Humboldt-Stiftung}

\author[A. Bobrowski]{Adam Bobrowski}
\address{Department of Mathematics, Faculty of Electrical Engineering
and Computer Science, Lublin University of Technology, Nadbystrzycka 38A 20-618 Lublin, Poland}
\email{a.bobrowski@pollub.pl}

\date{\today}

\begin{abstract}
Motivated by positivity-, monotonicity-, and convexity preserving differential equations, we introduce a definition of shape preserving operator semigroups and analyze their fundamental properties. In particular,
we prove that the class of shape preserving semigroups is preserved by perturbations and taking limits. These results are applied to partial delay differential equations.
\end{abstract}

\maketitle
\section{Introduction}

In applications it is often important to know whether some properties of  solutions of an evolution equation involved remain unchanged in time.
This is the case, for instance, with positivity of solutions of population equations, with monotonicity of solutions in transport processes and with convexity of solutions of certain partial differential equations, see for example Bian and Guan \cite{bian-guan},Blossey and Durran \cite{blossey-durran},Ekstrom and Tysk \cite{ekstrom-tysk}, Korevaar \cite{korevaar}, or P.--L.~Lions and Musiela \cite{lions-musiela}. Despite considerable interest in these questions, an abstract operator semigroup
theoretic approach seems to exist only in special situations (see Remark \ref{rem:qqriq}). The aim of this note is to look for the right definition of shape-preserving semigroups, and to present basic examples and properties of such semigroups. In particular, we would like to argue that a previous definition, due to M. Kov\'acs \cite{kovacs}, though it connects geometrical notions with semigroup theory in an intriguing way, has its deficiencies, and propose a way to mend them. We leave many important questions unresolved, but hope to set the research on the right track and stimulate some activity in this field.

As we shall see in Section \ref{stability}, with the modified definition of $S$-shape preservability, under natural assumptions, limits and $S$-shape preserving perturbations will not lead out of the class of $S$-shape preserving semigroups. This, in contradistinction to Kov\'acs' approach,  allows dealing with concrete examples of Section \ref{four}.

To recall Kov\'acs' definition,
let $X$ and $Y$ be Banach lattices, $S:X\to Y$ be a closed operator, and let $(A,D(A))$ be the generator of a strongly continuous semigroup $\sem{A}$ in $X$.
The archetypical examples of $S$ are the operators of first and second derivatives, describing monotone and convex functions, respectively.
More specifically, if $I\subset \RR$ is an open interval and $f:I\to \RR$ is sufficiently smooth, then $f$ is monotonically increasing if $f'\geq 0$, and
$f$ is convex if $f''\geq 0$.
The semigroup $\sem{A}$ is said to be \textbf{$S$-shape preserving in the sense of Kov\'{a}cs} if the following two conditions are satisfied:
\begin{enumerate}[(a)]
\item  $D(A) \subset D(S)$,
\item  $S\e^{tA} x \ge 0$ for $t \ge 0,$  provided $Sx \ge 0,\  x \in D(A).$
\end{enumerate}
(We note that taking $S$ to be the identity operator, we obtain the definition of a positive semigroup.)

We would like to argue that the following definition is more suitable. Let $X$ be a Banach space, $Y$ an ordered Banach space with closed positive cone, and let $S:X\to Y$ be a closed operator. Consider the sets
\begin{equation} \label{zbiorc}
C:= \overline{\{x \in D(S), Sx \ge 0 \}}, \qquad C_A := \overline {\{x \in D(A) \cap D(S), Sx \ge 0 \}}.
\end{equation}

\begin{definition}\label{def:shape_pres} We say that the semigroup $\sem{A}$ \textbf{preserves $S$-shape} if
\begin{itemize}
\item [(a)] $C= C_A$, and
\item [(b)] $\sem{A}$ leaves $C$ invariant.
\end{itemize}
\end{definition}
We comment that the most important part of the definition is that the semigroup leaves the set $C$ invariant. The equality in (a) is a technical condition, which we could not avoid in proving the invariance under Miyadera-Voigt type perturbations. However, in the examples presented later on, this condition was no restriction; on the contrary, it constituted a key technical tool in proving $S$-shape preservability. The requirement in question says that in a sense the semigroup generated by $A$ is \emph{compatible with $S$}, and it is our conviction that it agrees with intuitions better than Kovacs's condition (a).

\begin{remark}\label{rem:qqriq}
It is important to note that $C$ is a cone and the property just defined is a special case of the invariance of closed convex sets under semigroups. For the Hilbert space case, there is an extensive theory on this topic, and we thank an anonymous Referee for reminding this fact to us: for maximal monotone operators and nonlinear contraction semigroups, see Brezis \cite[Section IV.4]{brezis} and for operators defined by sesquilinear forms, see Ouhabaz \cite[Section 2.1]{ouhabaz}. Our approach differs from the ones just cited in the following aspects.
\begin{itemize}
\item We would like to use the special structure of the cone $C$, namely that it is given by an operator $S$.
\item We would like our results to be applicable in spacer other than Hilbert spaces, like the space of continuous functions.
\end{itemize}
Additionally, the criteria presented in the above two works seem to be difficult to apply in proving shape preservation, even in the Hilbert space case. For example, they require intimate knowledge of the projection onto the cone $C$.
\end{remark}

\section{Motivating examples}\label{examples}

\subsection*{Monotonicity}
In our first two examples,  $X = BUC(\mathbb R)$ is the space of bounded, uniformly continuous functions on $\mathbb R$. In order to describe monotone functions, we introduce $S = - \frac{\ud }{\ud x } $ with domain equal to $X^1$, the set of continuously differentiable functions in $X$ with the first derivative in $X.$

\begin{example}\label{eone} Recall, see Bertoin \cite[p. 11]{bertoin} or Kallenberg \cite[p.239]{kallenberg}, that a L\'evy process in $\mathbb R$ is a stochastic process $\xi_t, t \ge 0$ with stationary, independent increments and right-continuous paths with left-hand limits, and such that $P(\xi_0=0) = 1.$ The related semigroup of operators
$$ T_t f(x) = E \, f(x + \xi_t) \qquad x \in \mathbb R, t \ge 0$$
is strongly continuous in $X$, and $X^2$ (the set of twice continuously differentiable functions with both derivatives in $X$) is a core for its generator, see Bobrowski \cite[p. 279]{kniga}. (A special case is the shift semigroup). Moreover,  $T_t, t \ge 0$ leave $X^1$  invariant and
\begin{equation}
S T_t f = T_t Sf, \qquad f \in X^1, \label{monot}
\end{equation}
implying that the second condition in the definition of Kov\'{a}cs is satisfied. However, in general, the first one is not, as is seen from the case of the Poisson process, where the generator is bounded. At the same time, by \eqref{monot}, the semigroups related to L\'evy processes leave the set of non-increasing differentiable functions invariant, and Lemma \ref{lem:monot} in the Appendix implies that the same is true for the set of all non-increasing functions. Hence, the first condition of Kov\'acs seems to be too stringent, and contrary to intuition. On the other hand, the conditions of Definition \ref{def:shape_pres} are fulfilled.

\end{example}

\begin{example}
Consider a non-increasing function $\beta $ in $X$, and $A_\eps = \eps S + B, \eps >0$, where $B $ is the operator of multiplication by $\beta. $ The explicit Feynman--Kac-type formula
$$ \e^{tA_\eps } f(x) = \e^{\eps^{-1} \int_{x-\eps t}^x \beta (y) \ud y} f( x - \eps t) $$
makes it clear that all $\sem{A_\eps}, \eps >0$ preserve $S$-shape in the sense of Kov\'{a}cs. However, while the limit semigroup $\sem{B}$ (as $\eps \to 0$) given by
$$ \e^{tB} f (x) = \e^{t\beta(x)} f(x),$$
clearly maps non-increasing functions into non-increasing functions and is $S$-shape preserving in the sense of Definition \ref{def:shape_pres}, it does not satisfy condition (a) in Kov\'{a}cs' definition. This shows that this condition is not only contrary to intuition, but also that it causes $S$-shape preservability to be in general lost in the limit.
\end{example}

\begin{example}
Let $X=C[0,\infty]$ be the space of continuous functions on $\mathbb R^+$ with limits at infinity, and let $A= \frac 12 \frac{\ud^2}{\ud x^2} $ with domain composed of twice continuously differentiable functions with the second derivative in $X$ such that $f''(0)=0$, be the generator of the stopped (or: absorbed) Brownian motion, see Bobrowski \cite{kniga} or Liggett \cite{liggett}. Then $\sem{A}$ preserves monotonicity. To see this we recall the following relation (see e.g. Bobrowski \cite{cosine} or Liggett \cite{liggett})
$$ \e^{tA} f (x) = \e^{tA_0} \tilde f(x) , \qquad x \ge 0, t> 0$$
where $\e^{tA_0}f(x)= \frac 1{\sqrt{2\pi t}} \int_{-\infty}^\infty \e^{ - \frac {(x-y)^2}{2t}}f(x+y) \ud y  $ and $\tilde f$ is an extension of $f$ to $\mathbb R$, given by $\tilde f(x) = 2f(0) - f(-x), x \le 0.$ Hence, the claim follows as in  Example \ref{eone}, since $\tilde f $ is non-increasing if $f$ is.
\end{example}

\subsection*{Convexity-related notions}
\begin{example} Consider the left shift semigroup on $X=BUC(\RR^+)$; the generator is $Af=f'$ with domain composed of differentiable functions such that $f' \in X$. Clearly, convex and concave functions are preserved by this semigroup. However, defining the operator $Sf=f''$, $D(S)=\{f\in BUC(\RR^+)\cap C^2(\RR^+)\,:\, f''\in BUC(\RR^+)\}$, we see that Kovacs's condition (a) is not satisfied. On the other hand,  Lemma \ref{lem:convex} shows both that convex functions in $X$ may be conveniently described in terms of $S$, and that our semigroup preserves $S$-shape.

\end{example}

\begin{example} In $X=C[0,1]$, the operator $Af:=f''$ with domain $D(A):=\{f\in C^2[0,1]:\, f''(0)=f''(1)=0\}$, generates a Feller, analytic semigroup (compare e.g. Engel \cite{Eng:02a} and Liggett \cite[p. 17]{liggetts}).  Defining again $Sf:=f''$, $D(S):=C^2[0,1]$, we see that for $f\in D(A)$ with $f''\geq 0$,
\begin{equation}
S\e^{tA}f = A\e^{tA}f = \e^{tA}Af = \e^{tA} Sf \geq 0, \label{rach}
\end{equation}
because the semigroup $\sem{A}$ is positivity preserving. Hence,  $\sem{A}$ leaves $C_A$ invariant. By Lemma \ref{lem:convex} it follows that $C=C_A$, and our semigroup preserves convex functions.
\end{example}

We observe that the semigroup related to the heat equation with Neumann boundary conditions does not preserve convexity. To see this, note that the range of the semigroup is contained in the domain of the generator, but the only convex functions in the domain are the constant functions.

Finally, let us mention two examples where the choice of the space $Y$ is nontrivial.
\begin{example}
Consider the heat equation in $X= L^2(0,\pi)$ with Dirichlet boundary conditions. The related semigroup \sem{A} is not convexity preserving: for example, the function $f(x)=x^2$ (or any other positive convex function) is mapped into a positive function $g=\e^{tA} f\in D(A)$ (since \sem{A} is holomorphic), which cannot be convex because $g(0)=g(\pi)=0$ . However, Farag\'{o} and Pfeil  \cite{Farago-Pfeil} have proved that  if $f\in H^2(0,\pi)$ satisfies  $f\leq 0$ and $f'' \ge 0$, then $(\e^{tA}f)''\geq 0$ for $t\geq 0$.
To incorporate this example into our set-up, we introduce the space $Y:=L^2(0,\pi)\times L^2(0,\pi)$ and the operator $S:X\to Y$ given by
\begin{equation*}
Sf:=(-f,f''), \qquad  f\in H^2(0,\pi).
\end{equation*}
Then the cited result says that the set $\{f \in X: f\le0 \, , f \in H^2(0,\pi), f'' \ge 0\}$ is left invariant by $\sem{A}$, implying that the same is true for the closure of this set,  denoted by $C$ in agreement with Definition \ref{def:shape_pres}. It is easy to see that $C=C_A=\overline{\{f \in X: f\le0 \, , f \in D(A), f'' \ge 0\}},$ (it suffices to show that a negative linear function belongs to $C_A$) i.e. that \sem{A} preserves $S$-shape.

To interpret this result in other terms, we note that, by  Lemma \ref{lem:convex},
   the set $C$ is the closure of $\{f \in X; f \le 0, f \text{ convex}\}$,
since uniform convergence implies convergence in $X$.
On the other hand, \sem{A} being holomorphic, it maps $X$ into $D(A)\subset C[0,\pi]$. Hence, \sem{A} maps $C$ into $C\cap C[0,\pi]$. In particular, by Lemma \ref{lem:convex_lp}, it maps negative convex functions into negative convex functions.
\end{example}

\begin{example} We present a natural example where $Y$ is not a Banach lattice (compare our definition with the definition of Kov\'acs).
Let $\bar\Omega:=[0,1]\times [0,1]$, $X:=C_0(\bar\Omega)$, $Y=C(\bar\Omega,\CC^{2\times 2})$ be the space of continuous $(2\times 2)$-matrix-valued functions, and let $S$ be defined by $Sf:=f''$ (the Hessian matrix of $f$) and $D(S):=C_0(\bar\Omega)\cap C^2(\bar\Omega)$. We see that $Y$ is not a Banach lattice, though it is ordered pointwise through the ordering of positive definiteness. For $f\in D(S)$, convexity is characterized by the positive semidefinitenss of $f''(x)$ for all $x\in \bar\Omega$. Hence, by Remark \ref{rem:app_conv_square} we see that the set $C$ equals the convex functions in $X.$
\end{example}

We close this section by mentioning an important open problem of generalization of convexity preserving properties of the heat semigroup to higher dimensions.

\begin{problem}
Let $\Omega\subset \RR^N$ be a smooth convex set and consider the Dirichlet-heat semigroup. Is it true that negative convex functions are mapped into convex functions? Similarly, considering the heat equation with Wentzell boundary conditions as discussed in Engel \cite{Eng:02a}, is it true that convex functions are mapped into convex functions?
\end{problem}

See also the discussions in Korevaar \cite{korevaar} or in P.--L.~Lions and Musiela \cite{lions-musiela} related to this problem.

\section{Stability of $S$-shape preservability}\label{stability}

In this section, we investigate stability of shape preservability under approximations and perturbations. Though the results are quite straightforward corollaries of the definition, we list them in detail because they are of importance in applications. Throughout this section,  as in the definition of $S$-shape preservability, $X$ is a Banach space, $Y$ an ordered Banach space with closed positive cone, and $S:X\to Y$ is a closed linear operator.

We start with the Trotter-Kato approximation theorem.

\begin{theorem}
Assume that $(A,D(A))$ and $(A_n,D(A_n))$ are the generators of strongly continuous $S$-shape preserving operator semigroups, $\sem{A}$ and $\sem{A_n}\subset \cL(X)$, respectively, such that there are constants $M\geq 1$ and $\omega\in\RR$ such that $\|\e^{tA}\|,\, \|\e^{tA_n}\| \leq M\e^{\omega t}$, and that for some $\Re e\, \lambda>\omega$,
\begin{equation*}
(\lambda - A)^{-1} = \lim_{n\to \infty} (\lambda -A_n)^{-1} \qquad  (strongly).
\end{equation*}
Suppose that $A$ is compatible with $S$ (i.e., $C_A=C$). Then
$\sem{A}$ is $S$-shape preserving, as well.
\end{theorem}

\begin{proof}
By the Trotter-Kato approximation theorem (see \cite[Theorem III.4.8.]{Engel-Nagel}), we have $T_n(t)x\to T(t)x$ as $n\to\infty$ for all $x\in X$.
Hence,  $C$ (being closed) is invariant for $\sem{A}$, as claimed.
\end{proof}

The Chernoff product formula yields another application:

\begin{theorem}
Assume that $(A,D(A))$ is the generator of a $C_0$-semigroup $\sem{A}$ and that  $V:\RR^+\to\cL(X)$ is a strongly continuous family of operators such that there exists $M\geq 1$ and $\omega\in\RR$ with
\begin{equation*}
\left\|V(t)^n\right\|\leq M\e^{tn\omega},
\end{equation*}
and that there is a core $D\subset D(A)$ such that
\begin{equation*}
\exists\lim_{h\to 0}\frac{V(h)x-x}{h} = Ax.
\end{equation*}
If $C=C_A$ and $
V(t)C\subset C,$
then $\sem{A}$ is $S$-shape preserving.
\end{theorem}

\begin{proof}
By the Chernoff product formula \cite[Theorem III.5.2]{Engel-Nagel}, we have
\begin{equation*}
V(t/n)^n x\to \e^{tA} x
\end{equation*}
for all $x\in X$. Hence, the claim follows since $C$ is closed. \end{proof}

An important consequence is that certain time-discretizations of $S$-shape preserving semigroups are $S$-shape preserving. This is enormously important in numerical problems: if a differential equation preserves a quantity we aim for numerical methods preserving the same quantity. Such numerical methods are  called geometric integrators; see Hairer et al. \cite{Hairer-Lubich-Wanner} for the corresponding theory for ordinary differential equations.

\begin{corollary}
Assume that $(A,D(A))$ and $(B,D(B))$ generate $S$-shape preserving semigroups $\sem{A}$ and $\sem{B}$, and $\overline{(A+B, D(A+B))}$ generates a $C_0$-semigroup $(U(t))_{t\geq 0}$. If there exists $M\geq 1$, $\omega\in\RR$ such that
\begin{equation*}
\left\| \left (\e^{tA}\e^{tB}\right)^n \right\| \leq M \e^{nt\omega},
\end{equation*}
then $(U(t))_{t\geq 0}$ is $S$-shape preserving. Moreover, in each time-step the sequential and the Strang splittings are $S$-shape preserving, i.e., for each $x\in C$,
\begin{equation*}
u^{sq}:=\left(\e^{\frac{t}{n}A} \e^{\frac{t}{n}B}\right)^n x\in C,
\end{equation*}
and
\begin{equation*}
u^{St}:=\left(\e^{\frac{t}{2n}B} \e^{\frac{t}{n}A}\e^{\tfrac{t}{2n}B} \right)^n x\in C.
\end{equation*}
\end{corollary}

\begin{proof}
The stability for the Strang splitting follows from 
\cite[Lemma 2.3]{Csomos-Nickel}. The consistency for the sequential splitting is in 
\cite[Corollary III.5.8]{Engel-Nagel}, for the Strang splitting it is a straightforward modification.
\end{proof}
Note that the sequential splitting is usually referred to as the Lie product formula.

\bigskip
We turn now our attention to perturbation problems.

\begin{corollary}\label{corone}
Assume that $(A,D(A))$  generates an $S$-shape preserving semigroup $\sem{A}$ and that $B\in\cL(X)$ is a bounded operator leaving $C$ invariant. Then the semigroup generated by $(A+B,D(A))$ is $S$-shape preserving.
\end{corollary}

\begin{proof} Since $C$ is closed,
it is clear that the (semi-)group
$
\e^{tB}:=\sum_{n=0}^{\infty} \frac{(tB)^n}{n!}$
is $S$-shape preserving. The stability of the Lie product formula follows from the following considerations. Let $M\geq 1$, $\omega\in\RR$ be such that $\|\e^{tA} \| \leq M\e^{t\omega}$. Introduce the new, equivalent norm as in 
\cite[Lemma II.3.10]{Engel-Nagel} such that ${|\!|\!|}\e^{tA} {|\!|\!|}\leq \e^{t\omega}$. Since ${|\!|\!|}\e^{tB}{|\!|\!|}\leq \e^{t{|\!|\!|}B{|\!|\!|}}$, the statement follows by
\begin{equation*}
\left\| (\e^{\frac{t}{n}A}\e^{\frac{t}{n}B})^n x\right\| \leq \big{|\!\big|\!\big|} (\e^{\frac{t}{n}A}\e^{\frac{t}{n}B})^n x \big{|\!\big|\!\big|}\leq \e^{(\omega + {|\!|\!|}B{|\!|\!|})t}{|\!|\!|}x{|\!|\!|}\leq M\e^{(\omega + {|\!|\!|}B{|\!|\!|})t}\|x\|.
\end{equation*}
\end{proof}

\begin{example}
If $A$ is the generator of a L\'evy process semigroup in $X= BUC(\mathbb R)$ and $B$ is the multiplication operator related to a non-increasing function $\beta$ in $X$, then the semigroup generated by $A+ B$ preserves monotonicity.
\end{example}

We can relax the boundedness of $B$ to allow Miyadera-type perturbations.

\begin{theorem}\label{thm:miyadera}
Assume that $(A,D(A))$  generates an $S$-shape preserving semigroup $\sem{A}$ and that $B\in\cL(D(A),X)$ is such that for all $x\in D(A)\cap D(S)$ with $Sx\geq 0$, we have
\begin{equation*}
Bx\in C.
\end{equation*}
If further there is a $q\in(0,1)$ and $t_0>0$ such that
\begin{equation*}
\int_0^{t_0} \|B\e^{tA} x\|\ud t \leq q\|x\| \quad \text{ for all } x\in D(A),
\end{equation*}
then $(A+B,D(A))$ generates an $S$-shape preserving semigroup.
\end{theorem}

\begin{proof}
By the perturbation theorem of Miyadera-Voigt \cite[Theorem III.3.14 and Corollary III.3.16]{Engel-Nagel} $(A+B,D(A))$ generates a strongly continuous semigroup $(U(t))_{t\geq 0}$, which is given by the Dyson-Phillips series
\begin{equation*}
U(t)x = \sum_{n=0}^{\infty} U_n(t)x,
\end{equation*}
where $U_0(t)=\e^{tA}$ and
\begin{equation*}
U_n(t)x=\int_0^t U_{n-1}(t-s)B\e^{sA}x \ud s \quad \text{ for all } x\in D(A).
\end{equation*}
By induction argument, since $C_A$ is closed, we have $U_n(t)C_A\subset C_A$, hence $U(t)C_A\subset C_A$. Note that here we use heavily that $C_A=C$.
\end{proof}

\section{Delay equations}\label{four}

Since many physical processes depend on a former state of the system as well, they have to be described by \emph{partial delay differential equations} containing a term  depending on the \emph{history function}. Although these partial differential equations cannot be written as an abstract Cauchy problem on the original state space $X$, their solutions can be obtained by an operator semigroup on an appropriate function space (called \emph{phase space}). For a systematic treatment of the problem we refer to the monograph B\'{a}tkai and Piazzera \cite{Batkai-Piazzera}, which will be our main reference here.

\noindent Consider the \emph{abstract delay equation} in the following form (see, e.g., B\'{a}tkai and Piazzera \cite{Batkai-Piazzera}):
\begin{equation}
\left\{
\begin{aligned}
\frac{\mathrm{d}u(t)}{\mathrm{d}t}&=Bu(t)+\Phi u_t, \qquad t\ge 0, \\
u(0)&=x\in X, \\
u_0&=f\in\mathrm{L}^p\big([-1,0],X\big)
\end{aligned}
\right.
\tag{DE}
\label{delay}
\end{equation}
on the Banach space $X$, where $\big(B,D(B)\big)$ is a generator of a strongly continuous semigroup on $X$, $1< p<\infty$, and $\Phi:\mathrm{W}^{1,p}\big([-1,0],X\big) \to X$ is a bounded and linear operator. The \emph{history function} $u_t$ is defined by $u_t(\sigma):=u(t+\sigma)$ for $\sigma\in[-1,0]$.

Our main assumptions will be the following.
\begin{enumerate}
\item The operator $(B,D(B))$ is $S$-shape preserving.
\item There is $\eta\in BV([-1,0],\cL(X))$ such that
\begin{equation*}
\Phi f:=\int_{-1}^0 d\eta(s) f(s).
\end{equation*}
\item We have that $\eta(s)C\subset C$.
\end{enumerate}

We start with the following abstract statement.
\begin{theorem}
The solutions of the delay equation \eqref{delay} are $S$-shape preserving, i.e., for all initial values $x\in C$, $f\in \mathrm{L}^p([-1,0],C)$, we have that $u(t)\in C$.
\end{theorem}

\begin{proof}
In order to rewrite \eqref{delay} as an abstract Cauchy problem, we take the product space $\mathcal{E}:=X\times \mathrm{L}^p\big([-1,0],X\big)$ and the new unknown function as
\begin{equation}
t\mapsto\mathcal{U}(t):=\binom{u(t)}{u_t}\in\mathcal{E}.
\nonumber
\end{equation}
Then (\ref{delay}) can be written as an abstract Cauchy problem on the space $\mathcal{E}$ in the following way:
\begin{equation}
\left\{
\begin{aligned}
\frac{\mathrm{d}\mathcal{U}(t)}{\mathrm{d}t}&=\mathcal{A}\mathcal{U}(t), \qquad t\ge 0, \\
\mathcal{U}(0)&=\tbinom{x}{f}\in\mathcal{E},
\end{aligned}
\right.
\tag{$\mathcal{ACP}$}
\label{acp_delay}
\end{equation}
where the operator $\big(\mathcal{A},D(\mathcal{A})\big)$ is given by the matrix
\begin{equation}
\mathcal{A}:=\left(\begin{array}{cc} B & \Phi \\ 0 & \frac{d}{d\sigma} \end{array}\right)
\end{equation}
on the domain
\begin{equation}
D(\mathcal{A}):=\left\{\tbinom{x}{f}\in D(B)\times \mathrm{W}^{1,p}\big([-1,0],X\big): \ f(0)=x  \right\}.
\nonumber
\end{equation}
It is shown in B\'{a}tkai and Piazzera \cite[Corollary 3.5, Proposition 3.9]{Batkai-Piazzera} that the delay equation \eqref{delay} and the abstract Cauchy problem \eqref{acp_delay} are equivalent, i.e., they have the same solutions. More precisely, the first coordinate of the solution of \eqref{acp_delay} always solves \eqref{delay}. Due to this equivalence, the delay equation is well-posed if and only if the operator $\big(\mathcal{A},D(\mathcal{A})\big)$ generates a strongly continuous semigroup on the space $\mathcal{E}$.

Further, it was also shown in \cite{Batkai-Piazzera} that
\begin{equation*}
\cA=\cA_1+\cA_2,
\end{equation*}
where
\begin{equation*}
\mathcal{A}_1:=\left(\begin{array}{cc} B & 0 \\ 0 & \frac{d}{d\sigma} \end{array}\right),
\end{equation*}
with $D(\cA_1):=D(\cA)$, and
\begin{equation*}
\mathcal{A}_2:=\left(\begin{array}{cc} 0 & \Phi \\ 0 & 0 \end{array}\right)
\end{equation*}
with $D(\cA_2)=X\times \mathrm{W}^{1,p}([-1,0],X)$, and that $\cA_2$ satisfies the conditions of the Miyadera-Voigt perturbation theorem.

Defining $\mathcal{C}:=C\times \mathrm{L}^{p}([-1,0],C)$ and
\begin{equation*}
\mathcal{S}:=\left(\begin{array}{cc} S & 0 \\ 0 & S\otimes Id \end{array}\right)
\end{equation*}
mapping to $Y\times L^p([-1,0],Y)$, we see that $\cA_1$ and $\cA_2$ satisfy the conditions of Theorem \ref{thm:miyadera}.
\end{proof}

Using this abstract result, we are able to deal with a large class of partial differential equations with delay. As an illustration, we give here two examples.

\begin{corollary}
Consider the transport equation
\begin{align*}
\partial_t u(t,x) &= \partial_x u(t,x) + c u(t-\tau,x), \quad t\geq0,\, \ x\geq 0,\\
u(s,x) &= f(s,x),\quad s\in [-\tau,0], \quad x\geq 0,
\end{align*}
where $\beta:\RR^+\to \RR^+$ is a monotonically increasing bounded continuous function. If $f(s,\cdot)$ is an increasing (decreasing) function for all $s\in [-\tau,0]$, then $u(t,\cdot)$ is a monotonically increasing (decreasing) function for all $t\geq 0$.
\end{corollary}

\begin{corollary}
Consider the diffusion equation
\begin{align*}
\partial_t u(t,x) &= \partial^2_{xx} u(t,x) + \begin{cases} c u(t-\tau,x+\tfrac{1}{2}), \quad t\geq0,\, \ x\in [0,\tfrac{1}{2}],\\
 c u(t-\tau,x-\tfrac{1}{2}), \quad t\geq0,\, \ x\in [\tfrac{1}{2},1],\end{cases}\\
u(t,0) &= u(t,1)=0 \quad t\geq -\tau,\\
u(s,x) &= f(s,x),\quad s\in [-\tau,0], \quad x\in [0,1],
\end{align*}
where $c>0$. If $f(s,\cdot)$ is a negative convex function for all $s\in [-\tau,0]$, then $u(t,\cdot)$ is negative and convex for all $t\geq 0$.
\end{corollary}

\section{Appendix}

We show here that sets of monotone and convex functions in certain spaces, are closures of the positivity sets of the first and the second derivatives, respectively. The aim is to give a further justification of our definition of preservation of shape.
Though the following results seem to be a common knowledge 
we include them for the convenience of the reader. As it was kindly pointed out to us by the editor, an alternative, standard and natural way of proving the first two lemmas is by using convolution with positive test functions. Convolving a non-increasing (or: convex) function in $BUC$ with a positive test function we obtain a non-increasing (or: convex) $C^\infty$ function, and by taking an approximate identity of test functions, we approximate our initial function in the uniform norm.

\begin{lemma}\label{lem:monot} A non-increasing function in $BUC(\mathbb R)$ may be approximated by continuously differentiable, non-increasing functions in the same space.
\end{lemma}
\begin{proof}
Given numbers $a<b$ and $c\ge d$, we may find a non-increasing
differentiable function $g$ on $[a,b]$ such that $g(a)= c, g'(a)=0, g(b) = d$ and $g'(b)=0$; this may be achieved by stretching and translating  $g(x) = \cos x , x \in [0,\pi]. $  In particular, for any non-increasing $f\in C[a,b]$ with $f(a)=c$ and $f(b)=d$, we have $\|g- f\|_{C[a,b]} \le c - d.$

Given $\eps >0$ and a non-increasing $f\in X$ we may find reals $a $ and $b$ such that $f(-\infty) - f(a) < \eps $ and $f(b) - f(+\infty) < \eps .$ Next, we may find a  natural $n$ and points $a= a_1< a_2 < \cdots <a_n = b$ such that $f(a_{i}) - f(a_{i-1}) < \eps $. Then,
$$g = f(-\infty) \1_{(-\infty,a_0)} + \sum_{k=0}^{n} g_i  \1_{[a_i,a_{i+1})} + f(\infty) \1_{(a_{n+1}, \infty)}  $$
where $a_0= a_1 - 1 , a_{n+1} = a_n + 1 $, and  $g_i $ are defined on $[a_i,a_{i+1}]$ as  non-increasing, continuously differentiable functions satisfying $g(a_i)= f(a_i), g'(a_i)=0, g(a_{i+1})=f(a_{i+1})$ and $g'(a_{i+1})=0,$ satisfies $\sup_{x \in \mathbb R} |f(x) - g(x) | < \eps. $ \end{proof}

\begin{lemma}\label{lem:convex}
Let $-\infty< a<b < \infty$ and let $I=[a,b]$ or $I=[a,\infty)$ or $I=(-\infty,b]$. Then all convex functions in $BUC(I)$ can be approximated by twice continuously differentiable convex functions.
\end{lemma}

\begin{proof} We give the proof here for a finite interval; the infinite case can be handled similarly.
We approximate $f$  first by a piecewise linear function $g$: given  a natural $n$ and midpoints $a=x_0<x_1<\ldots <x_n=b$ we find a function $g$ such that $g$ is linear on each interval $[x_i,x_{i+1}]$, and $g(x_i)=f(x_i)$ at all points $x_i.$ Since $f$ is convex, so is $g$ and continuity of $f$ implies that given $\eps >0$ we may choose a sufficiently dense mesh of midpoints to make sure that the distance between $f$ and $g$ is less than a given $\eps. $ Moreover, since the number of points where $g$ is not differentiable is finite and at these points both functions are equal, we may smoothen $g$ out at these points without increasing the distance between the functions to find a twice continuously differentiable convex function within $\eps$ distance of $f$, as claimed. 
\end{proof}

\begin{remark}\label{rem:approx_d}
Note that the approximating function $g$ described above satisfies $g''(a)=g''(b)=0$. Additionally, if $f(a)=f(b)=0$  then $g(a)=g(b)=0,$ as well.
\end{remark}

\begin{remark}\label{rem:app_conv_square}
This argument can be generalized in a straightforward way to higher dimensions. Namely, we have the following. Let $\bar\Omega:=[0,1]\times [0,1]$ and assume that $f:\bar\Omega\to \RR$ is convex. Then it can be approximated uniformly by smooth convex functions.
\end{remark}

\begin{lemma}\label{lem:convex_lp}
Let $X=L^p(a,b)$ and \rm $D:=\overline{\{f\in X:\, f \text{ convex}\}}$\it . Then $D\cap C[a,b]$ is composed of convex functions.
\end{lemma}

\begin{proof}
Assume that $f_n\in X, n \ge 1$  are convex and $f_n\to f$ in $X$, where $f$ is convex. Then there is a subsequence $(n_k)_{k\ge 1}$ such that $f_{n_k}(x) \to f(x), $ as $k\to \infty$ for $x $ in a set $E\subset [a,b]$ of measure $b-a$. It follows that
$$ f(\alpha x + (1-\alpha) y) \le \alpha f(x) + (1-\alpha) f(y)$$
for all $x,y\in E$ and for all $\alpha\in E_{x,y}$, where $E_{x,y}\subset [0,1]$ is a set of measure $1$. By continuity of $f$ we conclude first that the same inequality is true for all $\alpha \in [0,1]$, and then for all $x,y \in [a,b].$
\end{proof}



\section*{Acknowledgments}
The European Union and the European Social Fund have provided financial support to the project under the grant agreement no. T\'{A}MOP-4.2.1/B-09/1/KMR-2010-0003. Supported by the OTKA grant Nr. K81403. A.~B\'atkai was further supported by the Alexander von Humboldt-Stiftung.

\end{document}